\documentclass{amsart}

\usepackage[all,cmtip]{xy}

\usepackage{color,graphicx,amssymb,latexsym,amsfonts,txfonts,amsmath,amsthm}
\usepackage{pdfsync}

\usepackage{hyperref}
\hypersetup{
    colorlinks=true,       
    linkcolor=blue,          
    citecolor=blue,        
    filecolor=blue,      
    urlcolor=blue           
}

\newtheorem{theo}{Theorem}                     
\newtheorem{propo}{Proposition}                  
\newtheorem{coro}{Corollary} 
\newtheorem{lemm}{Lemma}
\theoremstyle{remark}                  
\newtheorem{rema}{\bf Remark}
\newtheorem{exem}{\bf Example}

\begin{document}

\title{Hyperelliptic quotients of generalized Humbert curves}

\author{Rub\'en A. Hidalgo}
\address{Departamento de Matem\'atica y Estad\'{\i}stica\\
Universidad de La Frontera, Temuco, Chile}
\email{ruben.hidalgo@ufrontera.cl}

\thanks{Partially supported by projects Fondecyt 1190001}

\begin{abstract} 
A group $H \cong {\mathbb Z}_{2}^{n}$, $n \geq 3$, of conformal automorphisms of a closed Riemann surface $S$ such that $S/H$ has genus zero and exactly $(n+1)$ cone points is called a generalized Humbert group of type $n$, in which case, $S$ is called a generalized Humbert curve of type $n$. It is known that a generalized Humbert curve $S$ of type $n \geq 4$ is non-hyperelliptic and that it admits a unique generalized Humbert group $H$ of type $n$. We describe those subgroups $K$ of $H$, acting freely on $S$, such that $S/K$ is hyperelliptic. 
\end{abstract}

\maketitle

\section{Introduction}
There is a well known equivalence between the categories of closed Riemann surfaces and that of non-singular irreducible complex projective curves. By the uniformization theorem, every genus zero Riemann surface is biholomorphic to the Riemann sphere $\widehat{\mathbb C}$ which, in turns, corresponds to the complex projective line ${\mathbb P}^{1}$. Those of genus one are represented by the elliptic curves $E_{\lambda}:=\{[x:y:z] \in {\mathbb P}^{2}: y^{2}z=x(x-z)(x-\lambda z)\}$, where $\lambda \in {\mathbb C}\setminus \{0,1\}$. Hyperelliptic Riemann surfaces of genus $g \geq 2$ (those admitting a conformal automorphism of order two with $2g+2$ fixed points) can be described by the hyperelliptic curves $E_{\lambda_{1},\ldots,\lambda_{2g-1}}:=\{[x:y:z] \in {\mathbb P}^{2}: y^{2}z^{2g-1}=x(x-z)\prod_{j=1}^{2g-1}(x-\lambda_{j}z)\}$, where $\lambda_{j} \in {\mathbb C}\setminus \{0,1\}$ and $\lambda_{i} \neq \lambda_{j}$ for $i \neq j$ (this curve is singular at the point at infinity, so one needs to make a desingularization process to obtain a smooth one). 
For the non-hyperelliptic ones, necessarily of genus $g \geq 3$, each basis of its $g$-dimensional space of holomorphic one-forms produces a holomorphic embedding of it into ${\mathbb P}^{g-1}$ as a non-singular irreducible complex projective curve of degree $2g+2$, called a ``canonical" curve. For genus three, this canonical curve is a non-singular quartic plane curve and, for genus four (with some exceptions), they are a complete intersection of a cubic and a quadric hypersurfaces in ${\mathbb P}^{3}$. For $g \geq 5$, Petri's theorem \cite{Saint} asserts that, if the surface is non-trigonal and neither a plane quintic, then the canonical curve is a complete intesection of $(g-3)(g-2)/2$ quadric hypersurfaces. The first case is $g=5$, in which case we obtain three quadrics.

In \cite{Humbert}, Humbert described a two-dimensional family of degree $7$ genus five curves in ${\mathbb P}^{3}$, called classical Humbert curves. Later, in  \cite{Baker}, Baker rediscovered these curves related to a Weddle surface. Classical Humbert curves are exactly those 
closed Riemann surfaces $S$ admitting a (necessarily unique) group $H \cong {\mathbb Z}_{2}^{4}$ of conformal automorphisms with quotient orbifold $S/H$ of genus zero and with exactly five cone points (each one necessarily of order two).  Some facts about classical Humbert curves, mainly from the point of view of algebraic geometry, may be found in \cite{Accola, Edge, Varley}.  In \cite{Edge}, Edge observed that each classical Humber curve $S$ can be described as the complete intersection of three diagonal quadric hypersurfaces in ${\mathbb P}^{4}$: 
$$\left\{[x_{1}:x_{2}:x_{3}:x_{4}:x_{5}] \in {\mathbb P}^{4}: x_{1}^{2}+x_{2}^{2}+x_{3}^{2}+x_{4}^{2}+x_{5}^{2}=0,\right.$$
$$\left. a_{1}x_{1}^{2}+a_{2}x_{2}^{2}+a_{3}x_{3}^{2}+a_{4}x_{4}^{2}+a_{5}x_{5}^{2}=0, \; a_{1}^{2}x_{1}^{2}+a^{2}_{2}x_{2}^{2}+a_{3}^{2}x_{3}^{2}+a_{4}^{2}x_{4}^{2}+a_{5}^{2}x_{5}^{2}=0\right\},$$
where $a_{1},\ldots,a_{5} \in {\mathbb C}$ are different (in this model, $H=\langle a_{1},\ldots,a_{4}\rangle$, where $a_{j} \in {\rm PGL}_{5}({\mathbb C})$ is given by multiplication by $-1$ at coordinates $x_{j}$). We may identify 
the quotient orbifold $S/H$ with the Riemann sphere $\widehat{\mathbb C}$ and its five cone points (up to a M\"obius transformation) with $\infty$, $0$, $1$, $\lambda_{1}$ and $\lambda_{2}$. In \cite{CGHR}, it was observed that $S$ can be described by  the much simple set of quadrics (and the same representation for $H$ as above):
$$\left\{[x_{1}:x_{2}:x_{3}:x_{4}:x_{5}] \in {\mathbb P}^{4}: x_{1}^{2}+x_{2}^{2}+x_{3}^{2}=0, \; \lambda_{1} x_{1}^{2}+x_{2}^{2}+x_{4}^{2}=0, \; \lambda_{2}x_{1}^{2}+x_{2}^{2}+x_{5}^{2}=0\right\}.$$

Generalizing \cite{Edge}, in \cite{Edge2}, Edge considered curves obtained as the complete intersection of $n-1$ quadrics in ${\mathbb P}^{n}$. We are interested in some particular cases of them.

A closed Riemann surface $S$ is called a {\it generalized Humbert curve of type $n \geq 3$} if it admits a group ${\mathbb Z}_{2}^{n} \cong H \leq {\rm Aut}(S)$ such that the quotient orbifold $S/H$ has genus zero and exactly $(n+1)$ cone points (each one necessarily of order $2$). We say that $H$ is a {\it generalized Humbert group of type $n$} and that $(S,H)$ a {\it generalized Humbert pair of type $n$}.  In this case, by the Riemann-Hurwitz formula, $S$ has genus $g_{n}=1+2^{n-2}(n-3)$. In particular, for $n=3$ we obtain a genus one Riemann surface and $n=4$ corresponds to the classical Humbert curves.  If $n \geq 4$, then $g_{n} \geq 5$, $S$ is non-hyperelliptic \cite{Hidalgo} and $H$ is the unique generalized Humbert group of type $n$ of $S$ \cite{HKLP}.
In \cite{CHQ}, there were obtained an explicit isogenous decomposition of the jacobian variety $JS$, where its factors are either elliptic or the jacobian variety of certain hyperelliptic (branched) quotients of $S$ (each of them being a two-fold branched cover over $\widehat{\mathbb C}$ and whose branch values are contained inside the set of cone points of $S/H$).

Let $(S,H)$ be a generalized Humbert pair of type $n \geq 4$. A description of those non-trivial subgroups $K$ of $H$, acting freely on $S$, is given in Section \ref{libre} (Theorem \ref{lema1}) and also an algebraic description for the corresponing quotient $S/K$ is stated. 

Our first main result, Theorem \ref{explicito}, is a description of the subgroups $K$ of $H$, acting freely on $S$, and such that $S/K$ is hyperelliptic. For each possible situation, we provide a corresponding hyperelliptic algebraic equation.
In Section \ref{Ejemplo}, we make the above explicit for the case $n=4$ (i.e., for classical Humbert curves). 
Our second result, Theorem \ref{main}, describes some relations between different parametrizing spaces associated to these hyperelliptic surfaces and the moduli space of generalized Humbert curves.

\section{Generalized Humbert curves of type $n \geq 4$}\label{Sec:uniformiza}
In this section, $(S,H)$ will denote a generalized Humbert pair of type $n \geq 4$.  We may identify the quotient orbifold $S/H$ with the Riemann sphere $\widehat{\mathbb C}$ and its $n+1$ cone points, each one of order two, with a collection $\{\infty, 0, 1, \lambda_{1},\ldots, \lambda_{n-2}\} \subset \widehat{\mathbb C}$. That choice for the cone points is unique up to the action of elements of ${\rm PSL}_{2}({\mathbb C})$ sending three of them to $\infty, 0, 1$. Let $\pi:S \to \widehat{\mathbb C}$ be a regular (branched) cover map, with $H$ as its deck group, and whose branch values are the chosen cone points.

\subsection{Fuchsian uniformization} 
The uniformization theorem asserts the existence of 
a co-compact Fuchsian group $\Gamma_{n}$, of conformal automorphisms of the hyperbolic plane 
${\mathbb H}^{2}$, with a presentation
$\Gamma_{n}=\langle x_{1},..., x_{n+1}: x_{1}^{2}=\cdots=x_{n+1}^{2}=x_{1}x_{2}\cdots x_{n+1}=1\rangle$,
such that 
$S/H$ is conformally equivalent (as orbifolds) to  ${\mathbb H}^{2}/\Gamma_{n}$. As the derived subgroup  
$\Gamma'_{n}$ of $\Gamma_{n}$ is torsion free (by results due to Maclachlan \cite{Maclachlan}), the quotient
$X_{n}:={\mathbb H}^{2}/\Gamma'_{n}$ is a closed Riemann surface 
with $H_{n}:=\Gamma_{n}/\Gamma'_{n} \cong {\mathbb Z}_{2}^{n} < Aut({\mathbb H}^{2}/\Gamma'_{n})$ such that $X_{n}/H_{n}={\mathbb H}^{2}/\Gamma_{n}$ (so,  conformally equivalent to $S/H$). This permits to see that there is a biholomorphism $\phi:S \to X_{n}$ conjugating $H$ onto $H_{n}$ \cite{CGHR}.
This in particular asserts that any two generalized Fermat pairs of the same type are topologically equivalent.

\begin{rema}
Let $M<{\rm PSL}_{2}({\mathbb C})$ be the (finite) group of those M\"obius transformations keeping invariant the collection $\{\infty,0,1,\lambda_{1},\ldots,\lambda_{n-2}\}$. As $H$ is a normal subgroup of ${\rm Aut}(S)$ \cite{HKLP}, there is a natural homomorphism $\theta:{\rm Aut}(S) \to M$, whose kernel is $H$. As previously observed, 
 if $\Gamma<{\rm PSL}_{2}({\mathbb R})$ is a Fuchsian group such that $S/H={\mathbb H}^{2}/\Gamma$, then $S={\mathbb H}^{2}/\Gamma'$ and $H=\Gamma/\Gamma'$, where $\Gamma'$ is the derived subgroup of $\Gamma$. In particular, every (orbifold) automorphism of the orbifold $S/H$ lifts to an automorphism of $S$. So,  $\theta$ is also surjective and we have a short exact sequence
$1 \to H \to {\rm Aut}(S) \stackrel{\theta}{\to} M \to 1$, which permits to compute explicitly ${\rm Aut}(S)$ (see \cite{GHL}). 
The fact that $S$ is uniformized by $\Gamma'$ asserts that  $S$ is the highest abelian regular branched cover of the orbifold $S/H$ \cite{CGHR,GHL} (see also Section \ref{Sec:uniformiza}). So, if $R$ is a closed Riemann surface admitting an abelian group $G$ of conformal automorphisms such that there is some biholomorphism $\alpha:S/H \to R/G$ of orbifolds (i.e., a biholomorphism of the underlying Riemann surfaces structures sending the cone points to cone points and preserving their cone orders), then there is a subgroup $K \cong {\mathbb Z}_{2}^{m}$ of $H$, acting freely on $S$, and a biholomorphism $\beta:S/K \to R$ conjugating $H/K \cong {\mathbb Z}_{2}^{n-m}$ to $G$. 
\end{rema}

\subsection{Algebraic description}\label{algebra}
Let us consider the non-singular projective algebraic curve 
\begin{equation}\label{eq2}
C(\lambda_{1},...,\lambda_{n-2})=\left\{ \begin{array}{ccc}
x_{1}^{2}+x_{2}^{2}+x_{3}^{2}&=&0\\
\lambda_{1}x_{1}^{2}+x_{2}^{2}+x_{4}^{2}&=&0\\
\vdots \\
\lambda_{n-2}x_{1}^{2}+x_{2}^{2}+x_{n+1}^{2}&=&0
\end{array}
\right\} \subset {\mathbb P}^{n}.
\end{equation}

Each of the linear transformations $$a_{j}[x_{1}: \cdots :x_{n+1}]=[x_{1}: \cdots : x_{j-1} : -x_{j} : x_{j+1} : \cdots : x_{n+1}], \; j=1,\ldots, n+1,$$
induces, by restriction, an automorphism of $C(\lambda_{1},...,\lambda_{n-2})$ and, in particular,  the group $H_{0}=\langle a_{1},\ldots, a_{n}\rangle \cong {\mathbb Z}_{2}^{n}$ is a grup of automorphisms of it. We call $\{a_{1},\ldots,a_{n+1}\}$ is the set of the {\it standard generators} of $H_{0}$. Note that 
$a_{1}a_{2}\cdots a_{n}a_{n+1}=1$.   These standard generators are exactly those 
non-trivial elements of $H_{0}$ having fixed points on $C(\lambda_{1},...,\lambda_{n-2})$.
The rational map 
$\pi_{0}:C(\lambda_{1},...,\lambda_{n-2}) \to \widehat{\mathbb C}: [x_{1}: \cdots : x_{n+1}] \mapsto -(x_{2}/x_{1})^{2}$
defines a regular (branched) cover, with deck group $H_{0}$, whose branch values are $\{\infty,0,1,\lambda_{1},\ldots,\lambda_{n-2}\}$. It follows that $(C(\lambda_{1},...,\lambda_{n-2}),H_{0})$ is a generalized Humbert pair of type $n$. In \cite{CGHR} it was observed that there is a biholomorphism $\psi:S \to C(\lambda_{1},...,\lambda_{n-2})$ conjugating $H$ onto $H_{0}$ (and such that $\pi_{0} \circ \psi=\pi$).

If we denote by $Fix(a_{j})$ the locus of fixed points of $a_{j}$ on $C(\lambda_{1},...,\lambda_{n-2})$, then $Fix(a_{j})=C(\lambda_{1},...,\lambda_{n-2}) \cap \{x_{j}=0\}$, and
$\pi_{0}(Fix(a_{1}))=\infty$, $\pi_{0}(Fix(a_{2}))=0$, $\pi_{0}(Fix(a_{3}))=1$, $\pi_{0}(Fix(a_{j}))=\lambda_{j-3}$, for $j=4,...,n+1$.

\begin{rema}
As, for $n \geq 5$,  $n < g_{n}-1$, it holds that $C(\lambda_{1},\ldots,\lambda_{n-2})$ is not a canonical curve. Nevertheless, in \cite{Hidalgo:bases} it was noted that such a curve is a projection of a suitable canonical curve obtained by forgetting some coordinates.
\end{rema}

\subsection{Moduli of generalized Fermat curves}\label{Sec:moduli}
The connected open set 
$$V_{n}=\{(\lambda_{1},...,\lambda_{n-2}) \in {\mathbb C}^{n-2}: \lambda_{1},...,\lambda_{n-2} \in {\mathbb C}-\{0,1\}, \quad \lambda_{j} \neq \lambda_{r}, j \neq r\} \subset {\mathbb C}^{n-2},$$ 
admits the analytic automorphisms:
$$t(\lambda_{1},...,\lambda_{n-2})=\left(\frac{\lambda_{n-2}}{\lambda_{n-2}-1}, \frac{\lambda_{n-2}}{\lambda_{n-2}-\lambda_{1}},...,\frac{\lambda_{n-2}}{\lambda_{n-2}-\lambda_{n-3}}\right), \;
b(\lambda_{1},...,\lambda_{n-2})=\left(\frac{1}{\lambda_{1}},...,\frac{1}{\lambda_{n-2}}\right).$$

It is possible to note that ${\mathbb G}_{n}=\langle t,b\rangle \cong {\mathfrak S}_{n+1}$, the permutation on the $(n+1)$ cone points.

\begin{propo}[\cite{GHL}]
Let $\vec{\lambda}:=(\lambda_{1},...,\lambda_{n-2}), \vec{\delta}:=(\delta_{1},...,\delta_{n-2}) \in V_{n}$. Then the generalized Humbert curves 
$C(\lambda_{1},...,\lambda_{n-2})$ and $C(\delta_{1},...,\delta_{n-2})$ are conformally equivalent
 if and only if $\vec{\lambda}$ and $\vec{\delta}$ belong to the same ${\mathbb G}_{n}$-orbit. In particular, the quotient orbifold $V_{n}/{\mathbb G}_{n}$ is a model for the moduli space ${\mathcal H}_{n}$ of generalized Humbert curves of type $n$.
\end{propo}

\begin{rema}
If ${\mathcal T}(S)$ is the Teichm\"uller space of  $S$ and $Mod(S)$ is its modular group, then ${\mathcal M}_{g_{n}}={\mathcal T}(S)/{\rm Mod}(S)$ is its moduli space. Let us consider the homotopy class of $H$ inside ${\rm Mod}(S)$, which we still denoting by $H$, and let ${\mathcal T}_{H}(S) \subset {\mathcal T}(S)$ be the locus of its fixed points. If $N(H)<{\rm Mod}(S)$ is the normalizer of $H$ inside ${\rm Mod}(S)$, then $\widetilde{\mathcal H}_{n}={\mathcal T}_{H}(S)/N(H)$ is the normalization of the moduli space ${\mathcal H}_{n} \subset {\mathcal M}_{g_{n}}$. 
As noted above, any two generalized Humbert pairs, $(S_{1},H_{1})$ and $(S_{2},H_{2})$, of the same type $n$ are topologically equivalent, that is, there is an orientation-preserving homeomorphism $h:S_{1} \to S_{2}$ so that $h H_{1} h^{-1}=H_{2}$. It follows that $\widetilde{\mathcal H}_{n}$ is isomorphic to ${\mathcal H}_{n}$ (which is also isomorphic to ${\mathcal M}_{0,n+1}$, the moduli space of the $n+1$ punctured sphere) \cite{GHL}. 
\end{rema}

\section{Acting freely subgroups of the generalized Humbert group}\label{libre}
In this section, $S=C(\lambda_{1},\ldots,\lambda_{n-2})$ is a generalized Humbert curve of type $n \geq 4$, $H=H_{0}$ is its generalized Humbert group of type $n \geq 4$ (which is known to be unique \cite{HKLP}), and $\{a_{1},...,a_{n+1}\}$ its set of standard generators. In this section we provides a description of those non-trivial subgroups $K$ of $H$ acting freely on $S$ and equations for $S/K$.

\subsection{Subgroups acting freely}
For each $r \in \{1,\ldots,n-1\}$, set ${\mathbb Z}_{2}^{r}=\{u_{0}=1,u_{1},\ldots,u_{2^{r}-1}\}$ and let
${\mathcal F}_{r}^{n}$ be the collection of tuples  $(I_{1},\ldots,I_{2^{r}-1})$, where $\{I_{1},\ldots,I_{2^{r}-1}\}$ is a (disjoint) partition of the set $\{1,\ldots,n+1\}$ (some of the $I_{j}$ might be the empty set) such that: (i) $u_{1}^{\#I_{1}}\cdots u_{2^{r}-1}^{\#I_{2^{r}-1}}=1$ and (ii) $\langle u_{k}: I_{k}\neq \emptyset, k \in \{1,\ldots,2^{r}-1\}\rangle={\mathbb Z}_{2}^{r}$.

Note that, as $n+1 =\sum_{k=1}^{2^{r}-1} \#I_{k}$, it follows that, for $r=1,2$, at least one of the $I_{k}$ has cardinality at least two.
It might be that some ${\mathcal F}_{r}^{n}$ is the empty set. Set $A_{n}:=\{r \in \{1,\ldots,n-1\}:  {\mathcal F}_{r}^{n} \neq \emptyset\}$. 

If $r \in A_{n}$ and $P=(I_{1},\ldots,I_{2^{r}-1}) \in {\mathcal F}_{r}^{n}$, then we consider the  surjective homomorphism $\rho_{r}:H \to {\mathbb Z}_{2}^{r}$, defined by $\rho_{r}(a_{j})=u_{k}$ ($j \in I_{k}$), and set $K_{P} \cong {\mathbb Z}_{2}^{n-r}$ as its kernel. It can be seen that $K_{P}$ acts freely on $S$. 
This provides one direction of the next description.

\begin{theo}\label{lema1}
 If $\{I\} \neq K<H$ acts freely on $S$, then $K=K_{P}$ for suitable $r \in A_{n}$ and $P \in {\mathcal F}_{r}^{n}$.
\end{theo}
\begin{proof}
Each subgroup $K$ of $H$ is the kernel of a suitable surjective homomorphism $\rho:H \to {\mathbb Z}_{2}^{r}:=\{u_{0}=1,u_{1},\ldots,u_{2^{r}-1}\}$. If (i) $K$ is non-trivial, then $r \leq n-1$, and (ii) (as the only elements of $H$ acting with fixed points on $S$ are its standard generators) $K$ is torsion free if and only if $\rho(a_{j}) \neq 1$, for every $j=1,\ldots,n+1$ (in particular, $r \geq 1$). 
If we set $I_{k}=\{j \in \{1,\ldots,n+1\}: \rho(a_{j})=u_{k}\}$, for $k=1,\ldots,2^{r}-1$, then the condition (ii) ensures that $K$ is torsion free if
 $\{I_{1},\ldots,I_{2^{r}-1}\}$ is a (disjoint) partition of the set $\{1,\ldots,n+1\}$. Some of the $I_{k}$ might be the empty set.
  Also, as $a_{1}\cdots a_{n+1}=1$, we must have $u_{1}^{\#I_{1}}\cdots u_{2^{r}-1}^{\#I_{2^{r}-1}}=1$, and the surjectivity condition asserts $\langle u_{k}: I_{k}\neq \emptyset, k \in \{1,\ldots,2^{r}-1\}\rangle={\mathbb Z}_{2}^{r}$.\end{proof}

\begin{coro}\label{coro:m=n-1}
If $K \cong {\mathbb Z}_{2}^{n-1}$ is a subgroup of $H$, acting freely on $S$, then $n$ is odd. Moreover, in this situation, $K=\langle a_{1}a_{2}, a_{1}a_{3}, \ldots, a_{1}a_{n+1}\rangle$ and $S/K$ is the hyperelliptic Riemann surface defined by the curve $y^{2}z^{2g-1}=x(x-z)\prod_{j=1}^{n-2}(x-\lambda_{j}z)$.
\end{coro}
\begin{proof}
If $r=1$, then we are just considering the partition $I_{1}=\{1,\ldots,n+1\}$ and ${\mathbb Z}_{2}=\{1,u_{1}\}$. The condition $u_{1}^{\#I_{1}}=u_{1}^{n+1}=1$ is equivalent for $n$ to be odd. Now, under this condition on $n$, we obtain $K=\langle a_{1}a_{2}, a_{1}a_{3}, \ldots, a_{1}a_{n+1}\rangle$. As the Riemann surface $S/K$ is a two-fold branched cover of $S/H$, we obtain that $S$ is an hyperelliptic curve as described.
\end{proof}

\subsection{Algebraic curves}
Let $P:=(I_{1},\ldots,I_{2^{r}-1}) \in {\mathcal F}_{r}^{n}$, where $r \in A_{n}$, and $K_{P}$ the corresponding (non-trivial) subgroup of $H$ acting freely on $S$. Classical geometric invariant theory \cite{Mumford} permits to construct an algebraic curve model for $S/K_{P}$. For it, consider the affine model $S^{0}$ of $C(\lambda_{1},\ldots,\lambda_{n-2})$ obtained by setting the variable $x_{n+1}=1$. Next, we obtain a set of generators of the algebra of invariants ${\mathbb C}[x_{1},\ldots,x_{n}]^{K_{P}}$ (which is known to be finitely generated by Hilbert-Noether's theorem \cite{Noether1}). As the elements of $K_{P}$ are diagonal matrices, we may find such a set of generators formed by monomials. 
As a non-trivial element of $K_{P}$ is a diagonal matrix, whose diagonals is formed by values equal to $\pm 1$, some 
of these  invariant monomials are given by $t_{1}:=x_{1}^{2},\ldots, t_{n}:=x_{n}^{2}$,  and others will have the form $s_{k}=x_{1}^{l_{1,k}}x_{2}^{l_{2,k}}\cdots x_{n}^{l_{n,k}}$, for suitables $l_{j,k} \in \{0,1\}$, $k=1,\ldots,m_{P}$. 
In this case, if $\psi:S^{0} \to {\mathbb C}^{n+m_{P}}$ is defined by $\psi(x_{1},\ldots,x_{n})=(t_{1},\ldots,t_{n},s_{1},\ldots,s_{m_{P}})$, then the (possible singular) curve $\psi(S^{0})\subset {\mathbb C}^{n+m_{P}}$, which is defined by
\begin{equation}\label{curvita}
\left\{\begin{array}{c}
t_{1}+t_{2}+t_{3}=0, \; \lambda_{1}t_{1}+t_{2}+t_{4}=0, \ldots, \lambda_{n-3}t_{1}+t_{2}+t_{n}, \; \lambda_{n-2}t_{1}+t_{2}+1=0,\\
s_{k}^{2}=\prod_{j =1}^{n}t_{j}^{l_{j,k}}, \; k=1,\ldots,m_{P}.
\end{array}\right\},
\end{equation}
produces an affine model for $S/K_{P}$.  From the first line (of the above set of equations), we observe that 
$$t_{2}=-(1+\lambda_{n-2}t_{1}), \; t_{3}=1+(\lambda_{n-2}-1)t_{1}, \; t_{4}=1+(\lambda_{n-2}-\lambda_{1})t_{1}, \ldots, t_{n}=1+(\lambda_{n-2}-\lambda_{n-3})t_{1}.$$
Replacing the above in the equations on the second line of \eqref{curvita}, we obtain the following affine equations for $S/K_{P}$: 
\begin{equation}
\left\{\begin{array}{c}
s_{k}^{2}=\prod_{j =1}^{n}t_{j}^{l_{j,k}}; \; k=1,\ldots,m_{P}.
\end{array}\right\} \subset {\mathbb C}^{1+m}.
\end{equation}

\subsection{Some examples}\label{ejemplitos}

\subsubsection{}
Let $n \geq 4$ and $r=2$. We consider $P=(I_{1},I_{2},I_{3})$, where $\#I_{1}=n-1$.
\begin{enumerate}
\item If $n \geq 4$ is even,  $\#I_{2}=\#I_{3}=1$. Up to permutation of indices, we may assume 
$$I_{1}=\{1,\ldots,n-1\}, \; I_{2}=\{n\}, \; I_{3}=\{n+1\}.$$

\item If $n \geq 5$ is odd, $\#I_{2}=2$ and  $\#I_{3}=0$. Up to permutation of indices, we may assume 
$$I_{1}=\{1,\ldots,n-1\}, \; I_{2}=\{n,n+1\}.$$

\end{enumerate}

In any of these two cases, $K_{P}=\langle a_{1}a_{2},a_{1}a_{3},\ldots,a_{1}a_{n-1}\rangle \cong {\mathbb Z}_{2}^{n-2}$. The $K_{P}$-invariant monomials are $t_{1}=x_{1}^{2},\ldots,t_{n}=x_{n}^{2}$, $s_{1}=x_{1}x_{2}\cdots x_{n-1}$ and $s_{2}=x_{n}$. In this way, $S/K_{P}$ is the hyperelliptic Riemann surface (of genus $n-2$) defined by
$$S/K_{P}: \left\{\begin{array}{l}
s_{1}^{2}= -t_{1}(1+\lambda_{n-2}t_{1})(1+(\lambda_{n-2}-1)t_{1})\prod_{j=1}^{n-4}(1+(\lambda_{n-2}-\lambda_{j})t_{1})\\
s_{2}^{2}=1+(\lambda_{n-2}-\lambda_{n-3})t_{1}
\end{array}
\right\}.
$$

The group $G=H/K_{P} \cong {\mathbb Z}_{2}^{2}$ is generated by $A(t_{1},s_{1},s_{2})=(t_{1},-s_{1},s_{2})$ and $B(t_{1},s_{1},s_{2})=(t_{1},s_{1},-s_{2})$, where $A$ corresponds to the hyperelliptic involution.

\subsubsection{}
Let $n=5$, $r=2$ and $P=(I_{1},I_{2},I_{3})$ where $\#I_{j}=2$, for $j=1,2,3$. Up to permutation of the indices, assume 
$I_{1}=\{1,2\}, \; I_{2}=\{3,4\}, \; I_{3}=\{5,6\}.$
In this case, $K_{P}=\langle a_{1}a_{2}, a_{3}a_{4}, a_{1}a_{3}a_{5}\rangle \cong {\mathbb Z}_{2}^{3}$. The $K_{P}$-invariant monomials are $t_{1}=x_{1}^{2},\ldots,t_{5}=x_{5}^{2}$, $s_{1}=x_{1}x_{2}x_{5}$, $s_{2}=x_{3}x_{4}x_{5}$ and $s_{3}=x_{1}x_{2}x_{3}x_{4}$. In this way, $S/K_{P}$ is the Riemann surface of genus three (hyperelliptic if, for instance, $\lambda_{1}\lambda_{2}=\lambda_{3}$) defined by
$$S/K_{P}:\left\{\begin{array}{l}
s_{1}^{2}=-t_{1}(1+\lambda_{3}t_{1})(1+(\lambda_{3}-\lambda_{2})t_{1})\\
s_{2}^{2}=(1+(\lambda_{3}-1)t_{1})(1+(\lambda_{3}-\lambda_{1})t_{1})(1+(\lambda_{3}-\lambda_{2})t_{1})\\
s_{3}^{2}=-t_{1}(1+\lambda_{3}t_{1})(1+(\lambda_{3}-1)t_{1})(1+(\lambda_{3}-\lambda_{1})t_{1})
\end{array}
\right\}.
$$

The group $G=H/K_{P} \cong {\mathbb Z}_{2}^{2}$ is generated by $A(t_{1},s_{1},s_{2},s_{3})=(t_{1},-s_{1},s_{2},-s_{3})$ and $B(t_{1},s_{1},s_{2},s_{3})=(t_{1},s_{1},-s_{2},-s_{3})$. In this case, $A$, $B$ and $AB$ each one has $4$ fixed points.

In the particular case $\lambda_{3}=\lambda_{1}/\lambda_{2}$, it happens that $C(\lambda_{1},\lambda_{2},\lambda_{3})$ admits the conformal involution
$T([x_{1}:\cdots:x_{6}])= [x_{2}:\sqrt{\lambda_{1}} x_{1}: x_{4}: \sqrt{\lambda_{1}} x_{3}: \sqrt{\lambda_{2}} x_{6}: \sqrt{\lambda_{3}} x_{5}],$
which satisfies $\pi \circ T=\tau \circ \pi$, where $\tau(x)=\lambda_{1}/x$. If we chose the square roots such that $\sqrt{\lambda_{1}}=\sqrt{\lambda_{2}} \sqrt{\lambda_{3}}$, then $T^{2}=I$. As $T$ normalizes $K_{P}$, it induces a conformal involution $E$ of $S/K_{P}$, which is given by
$E(t_{1},s_{1},s_{2},s_{3})=(\lambda_{2} t_{2}/\lambda_{1}t_{5}, \lambda_{2}^{2}s_{1}/\lambda_{1} t_{5}^{2}, \lambda_{2}^{2} s_{2}/\lambda_{1} t_{5}^{2}, \lambda_{2}^{2}s_{3}/\lambda_{1} t_{5}^{2})$, where 
$t_{2}=-(1+\lambda_{3}t_{1})$ and $t_{5}=1+(\lambda_{3}-\lambda_{2})t_{1}$. It can be checked that $E$ is hyperelliptic involution of $S/K_{P}$.

\subsubsection{}
Let $n \geq 4$, $r=3$, ${\mathbb Z}_{2}^{3}=\langle u_{1}, u_{2}, u_{3}\rangle$ and $u_{4}=u_{1}u_{2}$, $u_{5}=u_{2}u_{3}$, $u_{6}=u_{1}u_{3}$ and $u_{7}=u_{1}u_{2}u_{3}$. We consider $P=(I_{1},\ldots,I_{7})$, where $\#I_{1}=n-2$ and
\begin{enumerate}
\item if $n \geq 4$ is even,  $\#I_{2}=\#I_{3}=\#I_{5}=1$ and $\#I_{4}=\#I_{6}=\#I_{7}=0$. Up to permutation of indices, we may assume 
$I_{1}=\{1,\ldots,n-2\}, \; I_{2}=\{n-1\}, \; I_{3}=\{n\}, \; I_{5}=\{n+1\}.$

\item if $n \geq 5$ is odd, $\#I_{2}=\#I_{3}=\#I_{7}=1$ and  $\#I_{4}=\#I_{5}=\#I_{6}=0$. Up to permutation of indices, we may assume 
$I_{1}=\{1,\ldots,n-2\}, \; I_{2}=\{n-1\}, \; I_{3}=\{n\}, \; I_{7}=\{n+1\}.$

\end{enumerate}

In any of these two cases, $K_{P}=\langle a_{1}a_{2},a_{1}a_{3},\ldots,a_{1}a_{n-2}\rangle \cong {\mathbb Z}_{2}^{n-3}$. The $K_{P}$-invariant monomials are $t_{1}=x_{1}^{2},\ldots,t_{n}=x_{n}^{2}$, $s_{1}=x_{1}x_{2}\cdots x_{n-2}$, $s_{2}=x_{n-1}$ and $s_{3}=x_{n}$. In this way, $S/K_{P}$ is the hyperelliptic Riemann surface (of genus $2n-5$) defined by
$$
S/K_{P}: \left\{\begin{array}{l}
s_{1}^{2}= -t_{1}(1+\lambda_{n-2}t_{1})(1+(\lambda_{n-2}-1)t_{1})\prod_{j=1}^{n-5}(1+(\lambda_{n-2}-\lambda_{j})t_{1})\\
s_{2}^{2}=1+(\lambda_{n-2}-\lambda_{n-4})t_{1}\\
s_{3}^{2}=1+(\lambda_{n-2}-\lambda_{n-3})t_{1}
\end{array}
\right\}.
$$

The group $G=H/K_{P} \cong {\mathbb Z}_{2}^{3}$ is generated by $A(t_{1},s_{1},s_{2},s_{3})=(t_{1},-s_{1},s_{2},s_{3})$, $B(t_{1},s_{1},s_{2},s_{3})=(t_{1},s_{1},-s_{2},s_{3})$ and $C(t_{1},s_{1},s_{2},s_{3})=(t_{1},s_{1},s_{2},-s_{3})$, where $A$ corresponds to the hyperelliptic involution.
\section{Hyperelliptic quotients}\label{hiper}
In this section, for $S=C(\lambda_{1},\ldots,\lambda_{n-2})$, where $n \geq 4$, 
we provide a description of those subgroups $K$ of $H$ such that $S/K$ is hyperelliptic.

\subsection{A known fact}
It is a well known fact that if  $R$ is a hyperelliptic Riemann surface admitting an abelian group $G \cong {\mathbb Z}_{2}^{r}$ of conformal automorphisms such that $R/G$ has genus zero, then $r\in\{1,2,3\}$ and, moreover, if the number of cone points is odd, then $r\neq 1$. Next, we write such a fact in terms of generalized Humbert curves and we provide a short argument for completeness.

\begin{propo}\label{cocientes}
Let $(S,H)$ be a generalized Humbert pair of type $n \geq 4$.  If $K \cong {\mathbb Z}_{2}^{m}$ is  a subgroup of $H$ acting freely on $S$ such that $R=S/K$ is hyperelliptic, then: 
{\rm (1)}  $m \in \{n-3,n-2,n-1\}$, if $n$ is odd, and {\rm (2)} $m \in \{n-3,n-2\}$, if $n$ is even.
\end{propo}
\begin{proof}
As $H \cong {\mathbb Z}_{2}^{n}$ acts with fixed points, $m \in \{1,...,n-1\}$. Note that, for $n \geq 4$ even, Theorem \ref{lema1} asserts that $H$ has no subgroup isomorphic to ${\mathbb Z}_{2}$ acting freely on $S$ (as noted before the same proposition, $r>1$ for $n$ even), in particular, 
$m \leq n-2$ in this case. The group $H/K \cong {\mathbb Z}_{2}^{n-m}$ is a group of automorphisms of the hyperelliptic Riemann surface $R$. If $\iota$ denotes the hyperelliptic involution of $R$, then either (i) $\iota \in H/K$ or (ii) $\iota \notin H/K$. In the first case, $H/K$ induces an action of
${\mathbb Z}_{2}^{n-m-1}$ as a group of M\"obius transformations. In the second case, $H/K$ induces an action of
${\mathbb Z}_{2}^{n-m}$ as a group of M\"obius transformations.
As the only Abelian subgroups of ${\rm PSL}_{2}({\mathbb C})$ are cyclic or ${\mathbb Z}_{2}^{2}$, the above asserts $m \in \{n-3,n-2,n-1\}$.
\end{proof}

\begin{rema}
We should observe that there are subgroups $K \cong {\mathbb Z}_{2}^{n-2}$, acting freely on $S$, such that $S/K$ is non-hyperelliptic (similarly for $K \cong {\mathbb Z}_{2}^{n-3}$). In fact, 
set $n=7$ and let $a_{1},\ldots,a_{8} \in H$ be the standard generators. Consider the surjective homomorphism $\rho:H \to G=\langle u,v: u^{2}=v^{2}=(uv)^{2}=1\rangle \cong {\mathbb Z}_{2}^{2}$ defined by $\rho(a_{i})=u$ and $\rho(a_{j})=v$, where $i \in \{1,2,3,4\}$ and $j \in \{5,6,7,8\}$. The kernel $K$ of $\rho$ acts freely on $S$ and $R=S/K$ is a closed Riemann surface of genus five on which the group $G$ acts as a group of conformal automorphisms with $S/H=R/G$. The involutions $u$ and $v$ have, each one, exactly $8$ fixed points and the involution $uv$ acts freely on $S$. We claim that $R$ is non-hyperelliptic. In fact, if $R$ is hyperelliptic, then its  hyperelliptic involution $\iota$ has $12$ fixed points, so $\iota \notin G$. By projecting $\iota$ to the quotient orbifold $R/G$, we see that the induced involution must have exactly two fixed points and the $8$ cone points are permuted into $4$ pairs. It follows that $R/\langle G,\iota\rangle$ has genus zero with exactly $6$ cone points of order two; exactly $4$ of them being the projection of the fixed points of the elements of $G$. Now, by projecting on the orbifold $R/\langle \iota \rangle$, the group $G$ induces an isomorphic M\"obius group that permutes the $12$ cone points (and fixing no one of them). This asserts that $R/\langle G, \iota \rangle$ must be of genus zero and with exactly $8$ cone points; where $3$ of them are the projections of the fixed points of $G$; a contradiction with the previous. 
\end{rema}

\subsection{Explicit descriptions for hyperelliptic situation}
As seen in Proposition \ref{cocientes}, if $K \cong {\mathbb Z}_{2}^{m}$ is an acting freely subgroups of $H$ such that $R=S/K$ is hyperelliptic, then $m \in \{n-3,n-2,n-1\}$ (where, for $n$ even, the case $n-1$ is not possible) over which the group $G=H/K \cong {\mathbb Z}_{2}^{n-m}$ acts as a group of conformal automorphisms with $R/G=S/H$. Next, we proceed to describe all these subgroups $K$ together the corresponding hyperelliptic algebraic equations (see also the examples in Section \ref{ejemplitos}).

\begin{theo}\label{explicito}
Let $S=C(\lambda_{1},...,\lambda_{n-2})$, where $n \geq 4$, and let $a_{1},\ldots,a_{n+1}$ be the standard generators of its generalized Humbert group $H \cong {\mathbb Z}_{2}^{n}$. Let $K$ be a subgroup of $H$, acting freely on $S$ and such that $R=S/K$ is hyperelliptic. Then one of the following hold.
\begin{enumerate}
\item $n \geq 5$ is odd, $K=\langle a_{1}a_{2},a_{1}a_{3},\ldots,a_{1}a_{n}\rangle \cong {\mathbb Z}_{2}^{n-1}$ and  $$R:= \; y^{2}=x(x-1)(x-\lambda_{1})\cdots(x-\lambda_{n-2}).$$

\item $n \geq 4$ is even, $K=\langle a_{i_{1}} a_{i_{2}}, \ldots, a_{i_{1}} a_{i_{n-1}}\rangle \cong {\mathbb Z}_{2}^{n-2}$,
where $\{i_{1}, \ldots, i_{n-1}\} \subset \{1,\ldots,n+1\}$.
If $\{b_{1}, b_{2}\} \in \{\infty,0,1,\lambda_{1},\ldots,\lambda_{n-2}\}$ are the projection of the fixed points of the two involutions in $\{a_{1},\ldots,a_{n+1}\} \setminus \{a_{i_{1}},\ldots,a_{i_{n-1}}\}$ and $Q(z)=b_{1}+b_{2}/z^{2}$, then  $$R: \; y^{2}=\prod_{j=1}^{2(n-1)}(x-\mu_{j}), \; 
\{\mu_{1},...,\mu_{2(n-1)}\}=Q^{-1}(\{\infty,0,1,\lambda_{1},...,\lambda_{n-2}\}-\{b_{1},b_{2}\}).$$ 

\item $n=5$, up to the ${\mathbb G}_{5}$-action we may assume $\lambda_{2}\lambda_{3}=\lambda_{1}$, $K=\langle a_{i_{1}}a_{i_{2}}, a_{i_{3}}a_{i_{4}}, a_{i_{1}}a_{i_{3}}a_{i_{5}} \rangle \cong {\mathbb Z}_{2}^{3}$,
where $\{1,2,3,4,5,6\}=\{i_{1},i_{2},i_{3},i_{4},i_{5},i_{6}\}$ and 
$$R:\; y^{2}=(x^{2}-a^{2})(x^{2}-a^{-2})(x^{2}-b^{2})(x^{2}-b^{-2}),$$
where $a^{2},b^{2} \in {\mathbb C}-\{0,1,-1\}$ are such that
$a^{2}+a^{-2}=4Q(2\sqrt{\lambda_{1}})$, $b^{2}+b^{-2}=4Q(-2\sqrt{\lambda_{1}})$,
$Q(x)=((x+\lambda_{1}/x)-1-\lambda_{1})/(\lambda_{2}+\lambda_{3}-\lambda_{1}-1)$ and $\lambda_{2}\lambda_{3}=\lambda_{1}$.

\item $K=\langle a_{i_{1}}a_{i_{2}},a_{i_{1}}a_{i_{3}},\ldots,a_{i_{1}}a_{i_{n-2}}\rangle \cong {\mathbb Z}_{2}^{n-3}$, where $\{i_{1},\ldots,i_{n-2}\} \subset \{1,\ldots,n+1\}$.
Let $\{b_{1},b_{2},b_{3}\} \subset \{\infty,0,1,\lambda_{1},\ldots,\lambda_{n-2}\}$ be the complement of the projection of the fixed points of $\{a_{i_{1}},\ldots,a_{i_{n-2}}\}$, $T(z)=(z-b_{2})(b_{3}-b_{1})/(z-b_{1})(b_{3}-b_{2})$, $U(z)=((1+z^{2})/2z)^{2}$ and $Q(z)=U \circ T^{-1}(z)$. If 
$\{\mu_{1},...,\mu_{n-2}\}=\{\infty,0,1,\lambda_{1},...,\lambda_{n-2}\}-\{b_{1},b_{2},b_{3}\}$, then
$$R: \quad y^{2}=\prod_{j=1}^{n-2} (x^{4}+2(1-2\mu_{j})x^{2}+1).$$
\end{enumerate}
\end{theo}

\subsection{A remark for $K \cong {\mathbb Z}_{2}^{n-2}$}
If $n \geq 5$ is odd, then there is exactly one subgroup $K \cong {\mathbb Z}_{2}^{n-1}$ of $H$ acting freely on $S$. If $n \geq 4$, then subgroups of $H$ isomorphic to ${\mathbb Z}_{2}^{n-2}$ acting freely on $S$ may not be unique. The following relates them (we use this in the proof of Theorem \ref{main}).

\begin{lemm}\label{teo5}
Let $(S,H)$ be a generalized Humbert pair of type $n \geq 4$ (not necessarily even). If $K_{1}$ and $K_{2}$ are two subgroups of $H$, both isomorphic to ${\mathbb Z}_{2}^{n-2}$ and both acting freely on $S$, then both pairs $(S/K_{1},H/K_{1})$ and $(S/K_{2},H/K_{2})$ are conformally equivalent if and only if there is some $f \in Aut(S)$ so that $f K_{1} f^{-1}=K_{2}$.
\end{lemm}
\begin{proof} One direction is clear by the uniqueness of $H$. On the other direction, 
if there is a conformal homeomorphism $\phi:S/K_{1} \to S/K_{2}$ so that $\phi (H/K_{1}) \phi^{-1} = H/K_{2}$, then $\phi$ induces a conformal automorphism $\psi$ of $S/H=(S/K_{j})/(H/K_{j})$. As $S$ is the homology cover of $S/H$, this means that $\psi$ lifts to a conformal automorphism $f \in Aut(S)$ that conjugates $K_{1}$ to $K_{2}$.
\end{proof}

In the generic situation one has that ${\rm Aut}(S)=H$ (recall that ${\rm Aut}(S)/H$ is isomorphic to the group of M\"obius transformations keeping invariant the set of conical points of $S/H$). In this case, the above result asserts that if  $K_{1}$ and $K_{2}$ are two different subgroups of $H$, both isomorphic to ${\mathbb Z}_{2}^{n-2}$ and acting freely on $S$, then $(S/K_{1},H/K_{1})$ and $(S/K_{2},H/K_{2})$ are not conformally equivalent.

\section{Proof of Theorem \ref{explicito}}
We set $r=n-m \in \{1,2,3\}$.
 Let us consider a regular branched covering $P:S \to R$, induced by the action of $K$. As $S$ has genus $1+2^{n-2}(n-3)$ and $K$ acts freely on $S$, by the Riemann-Hurwitz formula, $R$ has genus $g_{R}=1+2^{r-2}(n-3)$.

\subsection{Case ${\bf K\cong {\mathbb Z}_{2}^{n-1}}$}
As in this case $r=1$, Corollary \ref{coro:m=n-1} asserts that $n \geq 5$ is odd and that 
$$K=\langle a_{1}a_{2},a_{1}a_{3},\ldots,a_{1}a_{n}\rangle, \quad  S/K:= \; y^{2}=x(x-1)(x-\lambda_{1})\cdots(x-\lambda_{n-2}).$$ 

\subsection{Case ${\bf K\cong {\mathbb Z}_{2}^{n-2}}$}
In this case $r=2$, $R$ has genus $n-2$ and  
$G=\{1,u_{1}=u, u_{2}=v, u_{3}=uv\}=\langle u, v: u^{2}=v^{2}=(uv)^{2}=1\rangle$. Theorem \ref{lema1} asserts that $2 \in A_{n}$ and that there is some $P=(I_{1},I_{2},I_{3}) \in {\mathcal F}_{2}^{n}$ such that $K$ is the kernel  of a homomorphism $\rho_{2}:H \to G$ defined by  
$$\rho_{2}(a_{j})=\left\{\begin{array}{ll}
u, & j \in I_{1},\\
v, & j \in I_{2},\\
uv, & j \in I_{3}.
\end{array}
\right.
$$

As $P \in {\mathcal F}_{2}^{n}$, we must have that $1=u^{\#I_{1}}v^{\#I_{2}}(uv)^{\#I_{3}}=u^{\#I_{1}+\#I_{3}}v^{\#I_{2}+\#I_{3}}$, that is, 
the sets $I_{1}, I_{2}, I_{3}$ have cardinalities of the same parity (i.e., all of them odd or all of them even).

The group $G\cong {\mathbb Z}_{2}^{2}$ acts as a group of conformal automorphisms of $R$ such that $R/G=S/H$, that is, the Riemann sphere with exactly $n+1$ cone points of order two. The automorphism $u$ has $2\#I_{1}$ fixed points, $v$ has $2\#I_{2}$ fixed points and $uv$ has $2\#I_{3}$ fixed points.  Also, $\#I_{1}+\#I_{2}+\#I_{3}=n+1$. 

Let us denote by $\iota$ the hyperelliptic involution of $R$. Either (i) $\iota \in G$ or (ii) $\iota \notin G$.

\subsubsection{\bf Case ${\bf \iota \in G}$}
One of the elements of $G$ must be $\iota$, so one of $I_{j}$ must have cardinality $n-1$ ($R$ has genus $n-1$). 
As  $\#I_{1}+\#I_{2}+\#I_{3}=n+1$ and the three have the same parity, we have the following.
 \begin{itemize}
\item[(a)] If $n \geq 4$ is even, then either: (i) $\#I_{1}=n-1$ and $\#I_{2}=\#I_{3}=1$ or (ii) $\#I_{2}=n-1$ and $\#I_{1}=\#I_{3}=1$ or (iii) $\#I_{3}=n-1$ and $\#I_{1}=\#I_{2}=1$.
\item[(b)] If $n \geq 5$ is odd, then either: (i) $\#I_{1}=n-1$, $\#I_{2}=2$ and $\#I_{3}=0$ or (ii) $\#I_{1}=n-1$, $\#I_{2}=0$ and $\#I_{3}=2$ or
(iii) $\#I_{2}=n-1$, $\#I_{1}=2$ and $\#I_{3}=0$ or (iv) $\#I_{2}=n-1$, $\#I_{1}=0$ and $\#I_{3}=2$ or (v) $\#I_{3}=n-1$, $\#I_{1}=2$ and $\#I_{2}=0$ or (vi) $\#I_{3}=n-1$, $\#I_{1}=0$ and $\#I_{2}=2$.
\end{itemize}

The above permits to see that the corresponding collection of subgroups $K$ of $H$ (a collection of cardinality $n(n+1)/2$) are of the form
$$K_{P}=K(\{i_{1}, \ldots,i_{n-1}\})=\langle a_{i_{1}} a_{i_{2}}, \ldots, a_{i_{1}} a_{i_{n-1}}\rangle \cong {\mathbb Z}_{2}^{n-2},$$
where 
$\{i_{1}, \ldots, i_{n-1}\} \subset \{1,\ldots,n+1\}$ of cardinality $n-2$ (this corresponds to the component set $I_{j}$ of $P$ of cardinality $n-2$). 
Let $\{b_{1}, b_{2}\} \in \{\infty,0,1,\lambda_{1},\ldots,\lambda_{n-2}\}$ be the projection of the fixed points of the two involutions in $\{a_{1},\ldots,a_{n+1}\} \setminus \{a_{i_{1}},\ldots,a_{i_{n-1}}\}$. Consider the $2$-fold branched cover 
$Q:\widehat{\mathbb C} \to \widehat{\mathbb C}$, defined by $Q(z)=b_{1}+b_{2}/z^{2}$. The 
critical points of $Q$ are $\infty$ and $0$, it is even i.e., $Q(-z)=Q(z)$, and $\{Q(\infty),Q(0)\}=\{b_{1},b_{2}\}$. Then 

$$S/K(\{i_{1}, \ldots, i_{n-1}\}):= \; y^{2}=\prod_{j=1}^{2(n-1)}(x-\mu_{j}),$$
where $\{\mu_{1},...,\mu_{2(n-1)}\}=Q^{-1}(\{\infty,0,1,\lambda_{1},...,\lambda_{n-2}\}-\{b_{1},b_{2}\})$.

\subsubsection{\bf Case ${\bf \iota \notin G}$} We will observe that $n=5$.
As $\iota \notin G$, then $\iota$ does not share a fixed point with any involution of $G$. By 
projecting $\iota$ to $R/G$, we obtain a conformal involution $\widehat{\iota}$ that permutes in pairs the $n+1$ cone points $\{\infty,0,1,\lambda_{1},\ldots, \lambda_{n-2}\}$ and fixes none of them. It follows that $n$ is odd. Up to a M\"obius transformation, we may assume that $\widehat{\iota}$ permutes $\infty$ with $0$, $1$ with $\lambda_{1}$ (so $\widehat{\iota}(x)=\lambda_{1}/x$) and $\lambda_{2j}$ with $\lambda_{2j+1}$, for $j=1,\ldots, (n-3)/2$. This asserts that on the genus zero quotient $R/\langle G, \iota \rangle$ we have only two cone points coming from the fixed points of $\iota$ and the others $(n+1)/2$ from the fixed points of $G$.
We may now consider a regular branched cover of degree two induced by $\iota$, say $T:R \to \widehat{\mathbb C}$, so that $G$ induces, under $T$, the group $\widehat{G}=\langle A(x)=-x, B(x)=1/x\rangle$. The $2n-2$ branch values of $T$ are permuted by $\widehat{G}$, none of them being fixed by an involution on it. In particular, on the quotient $R/\langle \iota,G\rangle$ we must have that only three of the cone points coming from the fixed points of $G$ and the others $(2n-2)/4$ from the fixed points of $\iota$.
It follows that $2=(2n-2)/4$ and $(n-1)/2=3$, which asserts that $n=5$, that is, $S$ has genus $17$,
$R$ has genus three and $G \cong {\mathbb Z}_{2}^{2}$ acts with quotient $R/G$ being the sphere with $6$ cone points of order two. In this case, $R$ can be described by a hyperelliptic curve
$$y^{2}=(x^{2}-a^{2})(x^{2}-a^{-2})(x^{2}-b^{2})(x^{2}-b^{-2}),$$
where $a^{2},b^{2} \in {\mathbb C}-\{0,1,-1\}$ are such that
$a^{2}+a^{-2}=4Q(2\sqrt{\lambda_{1}})$, $b^{2}+b^{-2}=4Q(-2\sqrt{\lambda_{1}})$, 
$Q(x)=((x+\lambda_{1}/x)-1-\lambda_{1})/(\lambda_{2}+\lambda_{3}-\lambda_{1}-1)$ and $\lambda_{2}\lambda_{3}=\lambda_{1}$.
In this case, $G$ is generated by $u(x,y)=(-x,y)$ and $v(x,y)=(1/x,y/x^{4})\rangle$.
This situation corresponds to have $I_{1}=\{i_{1},i_{2}\}$, $I_{2}=\{i_{3},i_{4}\}$, $I_{3}=\{i_{5},i_{6}\}$ and $K=\langle a_{i_{1}}a_{i_{2}}, a_{i_{3}}a_{i_{4}},a_{i_{1}}a_{i_{3}}a_{i_{5}} \rangle \cong {\mathbb Z}_{2}^{3}$, 
where $\{1,2,3,4,5,6\}=\{i_{1},i_{2},i_{3},i_{4},i_{5},i_{6}\}$.

\subsection{Case ${\bf K\cong {\mathbb Z}_{2}^{n-3}}$}
In this case $r=3$, $R$ has genus $2n-5$ and 
$G=\{1,u_{1}=u, u_{2}=v, u_{3}=w, u_{4}=uv, u_{5}=vw, u_{6}=uw, u_{7}=uvw\}=\langle u,v,w: u^{2}=v^{2}=w^{2}=(uv)^{2}=(uw)^{2}=(vw)^{2}=1\rangle$, 
Theorem \ref{lema1} asserts that $3 \in A_{n}$ and that there is some $P=(I_{1},\ldots,I_{7}) \in {\mathcal F}_{3}^{n}$ such that $K$ is the kernel  of a homomorphism $\rho_{3}:H \to G$ defined by  
$$\rho_{3}(a_{j})=\left\{\begin{array}{ll}
u, & j \in I_{1},\\
v, & j \in I_{2},\\
w, & j \in I_{3},\\
uv, & j \in I_{4},\\
vw, & j \in I_{5},\\
uw, & j \in I_{6},\\
uvw, & j \in I_{7}.
\end{array}
\right.
$$

As $P \in {\mathcal F}_{3}^{n}$, we must have that
$$(*)\left\{ \begin{array}{c}
\#I_{1}+\#I_{4}+\#I_{6}+\#I_{7} \equiv 0 \mod(2),\\
\#I_{2}+\#I_{4}+\#I_{5}+\#I_{7} \equiv 0 \mod(2),\\
\#I_{3}+\#I_{5}+\#I_{6}+\#I_{7} \equiv 0 \mod(2),
\end{array}
\right.
$$

In this case, the group $G=H/K \cong {\mathbb Z}_{2}^{3}$ acts as a group of conformal automorphisms of $R$ such that $R/G=S/H$. The involutions $u,v,w,uv,vw,uw,uvw$ have, respectively, $4\#I_{1},4\#I_{2},4\#I_{3},4\#I_{4},4\#I_{5},4\#I_{6},4\#I_{7}$ fixed points.  

We are assuming $R$ to be hyperelliptic. We claim that its hyperelliptic involution $\iota$ must belong to $G$. In fact, if that is not the case, then $G \cong {\mathbb Z}_{2}^{3}$ must induce an isomorphic group of M\"obius transformations on the quotient $R/\langle \iota \rangle$. This is a contradiction to the fact that the only finite abelian subgroups of ${\rm PSL}_{2}({\mathbb C})$ are the cyclic ones and ${\mathbb Z}_{2}^{2}$. 

We may assume that $u=\iota$, that is, $\#I_{1}=n-2$; so $\#I_{2}+\#I_{3}+\#I_{4}+\#I_{5}+\#I_{6}+\#I_{7}=3$. It follows from this and  $(*)$ that 
$$\begin{array}{ll}
(a) & n-2+\#I_{4}+\#I_{6}+\#I_{7} \equiv 0 \mod(2),\\
(b) & \#I_{2}+\#I_{4}+\#I_{5}+\#I_{7},\in \{0,2\}\\
(c) & \#I_{3}+\#I_{5}+\#I_{6}+\#I_{7} \in \{0,2\}.
\end{array}
$$

If $\#I_{2}+\#I_{4}+\#I_{5}+\#I_{7} =0$, then (from (c)) $\#I_{3}+\#I_{6}\in \{0,2\}$, which contradicts the fact that  $\#I_{2}+\#I_{3}+\#I_{4}+\#I_{5}+\#I_{6}+\#I_{7}=3$. Similarly, 
if $\#I_{3}+\#I_{5}+\#I_{6}+\#I_{7} =0$, then it agains privides a contradiction. In this way,
$$\begin{array}{ll}
(i) & n-2+\#I_{4}+\#I_{6}+\#I_{7} \equiv 0 \mod(2),\\
(ii) & \#I_{2}+\#I_{3}+\#I_{4}+\#I_{5}+\#I_{6}+\#I_{7}=3,\\
(iii) & \#I_{2}+\#I_{4}+\#I_{5}+\#I_{7} =2,\\
(iv) & \#I_{3}+\#I_{5}+\#I_{6}+\#I_{7} =2.\\
\end{array}
$$

It follows, by combining (ii) and (iii) and then (ii) with (iv), that $\#I_{3}+\#I_{6}=1=\#I_{2}+\#I_{4}$.
Then, by (ii) one also have that $\#I_{5}+\#I_{7}=1$ and, by (i) that  $n-2+\#I_{4}+\#I_{6}+\#I_{7}$ is even. As (by (ii)) 
$\#I_{4}+\#I_{6}+\#I_{7} \in \{0,1,2,3\}$, we observe that, 
for $n$ even, $\#I_{4}+\#I_{6}+\#I_{7}\in \{0,2\}$ and, for $n$ odd, $\#I_{4}+\#I_{6}+\#I_{7} \in \{1,3\}$.

Summarizing all the above:
\begin{enumerate}
\item If $n \geq 4$ is even, then $\#I_{1}=n-2$, and either:
\begin{enumerate}
\item $\#I_{4}=\#I_{6}=\#I_{7}=0$ and $\#I_{2}=\#I_{3}=\#I_{5}=1$.
\item $\#I_{4}=\#I_{3}=\#I_{5}=0$ and $\#I_{2}=\#I_{6}=\#I_{7}=1$.
\item $\#I_{2}=\#I_{5}=\#I_{6}=0$ and $\#I_{3}=\#I_{4}=\#I_{7}=1$.
\item $\#I_{2}=\#I_{3}=\#I_{7}=0$ and $\#I_{4}=\#I_{5}=\#I_{6}=1$.

\end{enumerate}

\item If $n \geq 5$ is odd, then $\#I_{1}=n-2$, and either:
\begin{enumerate}
\item  $\#I_{2}=\#I_{6}=\#I_{7}=0$ and $\#I_{4}=\#I_{3}=\#I_{5}=1$.
\item  $\#I_{3}=\#I_{4}=\#I_{7}=0$ and $\#I_{2}=\#I_{5}=\#I_{6}=1$.
\item $\#I_{4}=\#I_{5}=\#I_{6}=0$ and $\#I_{2}=\#I_{3}=\#I_{7}=1$. 
\item  $\#I_{2}=\#I_{3}=\#I_{5}=0$ and $\#I_{4}=\#I_{6}=\#I_{7}=1$.
\end{enumerate}
\end{enumerate}

In each of the the above cases (a)-(d), for either $n$ even or odd, we have $n(n^{2}-1)/6$ possible tuples $P=(I_{1},\ldots,I_{7}) \in {\mathcal F}_{3}^{n}$, and for each of them we have the corresponding group $K=K_{P}$. 
If, for such a tuple $P$, we have $I_{1}=\{i_{1},\ldots,i_{n-2}\}$ and let $\{b_{1},b_{2},b_{3}\} \subset \{\infty,0,1,\lambda_{1},\ldots,\lambda_{n-2}\}$ be the complement of the projection of the fixed points of $\{a_{i_{1}},\ldots,a_{i_{n-2}}\}$, then  
$K_{P}=\langle a_{i_{1}}a_{i_{2}},a_{i_{1}}a_{i_{3}},\ldots,a_{i_{1}}a_{i_{n-2}}\rangle$.

\begin{rema}
Note, in the above, that different tuples determine the same $K$ if and only if the corresponding $I_{1}$ coincide.
\end{rema}

Let 
$T(z)=(z-b_{2})(b_{3}-b_{1})/(z-b_{1})(b_{3}-b_{2})$, $U(z)=((1+z^{2})/2z)^{2}$ and $Q(z)=U \circ T^{-1}(z)$. Then, $Q:\widehat{\mathbb C} \to \widehat{\mathbb C}$ is a regular branched cover with deck group $J=\langle z \mapsto -z, z \mapsto 1/z\rangle \cong {\mathbb Z}_{2}^{2}$ whose branch values are $b_{1}$, $b_{2}$ and $b_{3}$. Now, let us consider the $4n-8$ preimages under $Q$ of the points in $\{\mu_{1},...,\mu_{n-2}\}=\{\infty,0,1,\lambda_{1},...,\lambda_{n-2}\}-\{b_{1},b_{2},b_{3}\}$. The set of these lifted points by $Q$ is a disjoint union of $n-2$ sets of cardinality $4$ each one (each one is a complete orbit under $J$). The equation of the hyperelliptic curve $S/K_{P}$ is
$$S/K_{P}: \quad y^{2}=\prod_{j=1}^{n-2} (x^{4}+2(1-2\mu_{j})x^{2}+1).$$

\section{Example: $n=4$ (classical Humbert curves)} \label{Ejemplo}
Let $(S=C(\lambda_{1},\lambda_{2}),H=H_{0})$, where $(\lambda_{1},\lambda_{2}) \in V_{4}$, be a generalized Humbert pair of type $n=4$, and let
$\{a_{1}, a_{2}, a_{3}, a_{4}, a_{5}\}$ be the set of standard generators of $H$. In this case, $S$ has genus $g=5$ and 
the subgroups of $H$, acting freely and providing hyperelliptic quotients, are isomorphic to either ${\mathbb Z}_{2}$ or ${\mathbb Z}_{2}^{2}$.

\subsection{}
The $10$ subgroups of $H$, isomorphic to ${\mathbb Z}_{2}$ and acting freely on $S$, are given by
$$\begin{array}{lllll}
L_{1}=\langle a_{1}a_{2}\rangle,& L_{2}=\langle a_{1}a_{3}\rangle,& L_{3}=\langle a_{1}a_{4}\rangle,& L_{4}=\langle a_{1}a_{5}\rangle,& L_{5}=\langle a_{2}a_{3}\rangle,\\
L_{6}=\langle a_{2}a_{4}\rangle, & L_{7}=\langle a_{2}a_{5}\rangle, & L_{8}=\langle a_{3}a_{4}\rangle, & L_{9}=\langle a_{3}a_{5}\rangle, & L_{10}=\langle a_{4}a_{5}\rangle.
\end{array}
$$

The $10$ hyperelliptic curves of genus three, provided by these $10$ subgroups, are given by
$$y^{2}=(x^{4}+2(1-2a)x^{2}+1)(x^{4}+2(1-2b)x^{2}+1),$$
where $(a,b)$ runs over the following pairs
$$\begin{array}{l}
(\lambda_{1},\lambda_{2}),(1-\lambda_{1},\lambda_{2}(1-\lambda_{1})/(\lambda_{2}-\lambda_{1})),(\lambda_{1}/(\lambda_{1}-1),(\lambda_{2}-\lambda_{1})/(1-\lambda_{1})), \\
(1/\lambda_{1},\lambda_{2}/\lambda_{1}),  (1-\lambda_{2},\lambda_{1}(1-\lambda_{2})/(\lambda_{1}-\lambda_{2}),(\lambda_{2}/(\lambda_{2}-1),(\lambda_{1}-\lambda_{2})/(1-\lambda_{2})),\\
(1/\lambda_{2},\lambda_{1}/\lambda_{2}),
((1-\lambda_{1})/(1-\lambda_{2}),\lambda_{2}(1-\lambda_{1})/(\lambda_{1}(1-\lambda_{2}))),
(\lambda_{2}/\lambda_{1},(1-\lambda_{2})/(1-\lambda_{1})), \\ (\lambda_{1}/\lambda_{2},\lambda_{1}(1-\lambda_{2})/(\lambda_{2}(1-\lambda_{1}))).
\end{array}
$$

\subsection{}
The $10$ subgroups of $H$, isomorphic to ${\mathbb Z}_{2}^{2}$ and acting freely on $S$, are given by
$$\begin{array}{c}
K_{1}=\langle a_{1}a_{2},a_{1}a_{3}\rangle,\; K_{2}=\langle a_{1}a_{2},a_{1}a_{4}\rangle,\; K_{3}=\langle a_{1}a_{2},a_{1}a_{5}\rangle,\; K_{4}=\langle a_{1}a_{3},a_{1}a_{4}\rangle,\\
K_{5}=\langle a_{1}a_{3},a_{1}a_{5}\rangle,\;K_{6}=\langle a_{1}a_{4},a_{1}a_{5}\rangle, \; K_{7}=\langle a_{2}a_{3},a_{2}a_{4}\rangle,\\
K_{8}=\langle a_{2}a_{3},a_{2}a_{5}\rangle, \; K_{9}=\langle a_{2}a_{4},a_{2}a_{5}\rangle, \; K_{10}=\langle a_{3}a_{4},a_{3}a_{5}\rangle.
\end{array}
$$

In order to get algebraic curves descriptions, for the above corresponding $10$ Riemann surfaces of genus two, we proceed as follows. We consider the $10$ choices for $\{b_{1},b_{2}\}$: 
(i) $\{\infty,0\}$, (ii) $\{\infty,1\}$, (iii) $\{\infty,\lambda_{1}\}$, (iv) $\{\infty,\lambda_{2}\}$, (v) $\{0,1\}$, (vi) $\{0,\lambda_{1}\}$, (vii) $\{0,\lambda_{2}\}$, (viii) $\{1,\lambda_{1}\}$, (ix) $\{1,\lambda_{2}\}$, (x) $\{\lambda_{1},\lambda_{2}\}$. The choices for $Q(z)$ we may use in each case are: (i) $Q(z)=z^{2}$, (ii) $Q(z)=z^{2}+1$, (iii) $Q(z)=z^{2}+\lambda_{1}$,    (iv) $Q(z)=z^{2}+\lambda_{2}$,  (v) $Q(z)=1/(z^{2}+1)$, (vi) $Q(z)=\lambda_{1}/(z^{2}+1)$, (vii) $Q(z)=\lambda_{2}/(z^{2}+1)$,(viii) $Q(z)=(z^{2}+\lambda_{1})/(z^{2}+1)$,  (ix) $Q(z)=(z^{2}+\lambda_{2})/(z^{2}+1)$, (x) $Q(z)=(\lambda_{1}z^{2}+\lambda_{2})/(z^{2}+1)$. In this way, we obtain the $10$ desired hyperelliptic Riemann surfaces (in the first one, $C_{1}$, we have also changed $(x,y)$ by $(ix,iy)$):
$$
{\small
\begin{array}{l}
C_{1}: \; y^{2}=(x^{2}+1)(x^{2}+\lambda_{1})(x^{2}+\lambda_{2}),\;
C_{2}: \; y^{2}=(x^{2}+1)(x^{2}+1-\lambda_{1})(x^{2}+1-\lambda_{2}),\\
C_{3}: \; y^{2}=(x^{2}+\lambda_{1})\left(x^{2}-1+\lambda_{1}\right)\left(x^{2}-\lambda_{2}+\lambda_{1}\right),\;
C_{4}: \; y^{2}=(x^{2}+\lambda_{2})\left(x^{2}-1+\lambda_{2}\right)\left(x^{2}-\lambda_{1}+\lambda_{2}\right),\\
C_{5}: \; y^{2}=(x^{2}+1)\left(x^{2}+(\lambda_{1}-1)/\lambda_{1}\right)\left(x^{2}+(\lambda_{2}-1)/\lambda_{1}\right),\\
C_{6}: \; y^{2}=(x^{2}+1)\left(x^{2}+1-\lambda_{1}\right)\left(x^{2}+(\lambda_{2}-\lambda_{1})/\lambda_{2}\right),\\
C_{7}: \; y^{2}=(x^{2}+1)\left(x^{2}+1-\lambda_{2}\right)\left(x^{2}+(\lambda_{1}-\lambda_{2})/\lambda_{1}\right),\\
C_{8}: \; y^{2}=(x^{2}+1)(x^{2}+\lambda_{1})\left(x^{2}+(\lambda_{2}-\lambda_{1})/(1-\lambda_{2})\right),\\
C_{9}: \; y^{2}=(x^{2}+1)(x^{2}+\lambda_{2})\left(x^{2}+(\lambda_{1}-\lambda_{2})/(1-\lambda_{1})\right),\\
C_{10}: \; y^{2}=(x^{2}+1)\left(x^{2}+\lambda_{2}/\lambda_{1}\right)\left(x^{2}+(\lambda_{2}-1)/(\lambda_{1}-1)\right).
\end{array}
}
$$

Note that if we change $(x,y)$ by $\left(\sqrt{\lambda_{1}}x,\sqrt{\lambda_{1}^{3}}y\right)$, then $C_{3}$ is transformed into the curve
$$C'_{3}: \; y^{2}=(x^{2}+1)\left(x^{2}+(\lambda_{1}-1)/\lambda_{1}\right)\left(x^{2}+(\lambda_{1}-\lambda_{2})/\lambda_{1}\right)$$
and if we change $(x,y)$ by $\left(\sqrt{\lambda_{2}}x,\sqrt{\lambda_{2}^{3}}y\right)$, then $C_{4}$ is transformed into the curve
$$C'_{4}: \; y^{2}=(x^{2}+1)\left(x^{2}+(\lambda_{2}-1)/\lambda_{2}\right)\left(x^{2}+(\lambda_{2}-\lambda_{1})/\lambda_{2}\right)$$

\subsection{}
Each subgroup $K_{j}$ contains exactly $3$ of the subgroups $L_{k}$'s; for instance, $K_{1}$ contains $L_{1}$, $L_{2}$ and $L_{5}$.
As noted before, the genus two surface $S/K_{j}$ is obtained by considering two points $b_{1},b_{2} \in \{\infty,0,1,\lambda_{1},\lambda_{2}\}$. A Riemann surface $S/L_{k}$ over $S/H_{j}$ is obtained by considering a point $b_{3} \in \{\infty,0,1,\lambda_{1},\lambda_{2}\}-\{b_{1},b_{2}\}$. In this way, once we have chosen $b_{1}$ and $b_{2}$, there are exactly $3$ possible choices for $b_{3}$; these are the three subgroups $L_{k}$'s contained inside $K_{j}$.
For example, if we take $\{b_{1},b_{2}\}=\{\lambda_{1},\lambda_{2}\}$, 
then the genus two surface (uniformized by one of the $K_{j}$'s) is given by 
$$y^{2}=(x^{2}+1)\left(x^{2}+\lambda_{2}/\lambda_{1}\right)\left(x^{2}+(\lambda_{2}-1)/(\lambda_{1}-1)\right),$$
and the three genus three surfaces (uniformized by one of the $L_{k}$'s contained in the corresponding $K_{j}$) are 
$$
\begin{array}{ll}
y^{2}=(x^{4}+1)\left(x^{4}+\lambda_{2}/\lambda_{1}\right), &\mbox{if} \; b_{3}=1.\\
y^{2}=(x^{4}+\lambda_{2}/\lambda_{1})\left(x^{4}+(\lambda_{2}-1)/(\lambda_{1}-1)\right), & \mbox{if} \; b_{3}=\infty.\\
y^{2}=(x^{4}+1)\left(x^{4}+(\lambda_{2}-1)/(\lambda_{1}-1)\right), & \mbox{if} \; b_{3}=0.
\end{array}
$$

\section{A connection to some parametrizing spaces}\label{prueba}
\subsection{Some moduli spaces}
Some general facts on the complex analytical theory of the (coarse) moduli spaces of Riemann orbifolds can be found, for instance, in \cite{Nag, Royden}. We denote by
${\mathcal M}_{g}$ the moduli space of closed Riemann surfaces of genus $g \geq 1$. This is a complex orbifold of dimension (i) $3(g-1)$ for $g \geq 2$ and (ii) $1$ for $g=1$. The moduli space  of hyperelliptic Riemann surfaces of genus $g \geq 2$, which we denote by ${\mathcal M}_{g}^{hyp}$, is a complex orbifold of dimension $2g-1$. The uniqueness of the hyperelliptic involution asserts that there is a natural holomorphic embedding of ${\mathcal M}_{g}^{hyp}$ into ${\mathcal M}_{g}$. 

If $g \geq 2$ is even, then ${\mathcal M}_{(g,2)}$ denotes the sublocus of ${\mathcal M}_{g}$ consisting 
of those classes of Riemann surfaces admitting a conformal involution with exactly two fixed points. This is a complex suborbifold of dimension $(3g-2)/2$. 
In this case, each connected component of ${\mathcal M}^{hyp}_{(g,2)}:={\mathcal M}_{(g,2)} \cap {\mathcal M}_{g}^{hyp}$ can be identified with the moduli space ${\mathcal M}_{0,g+3}$ of $(g+3)$-marked spheres, which has dimension $g$.

Similarly, for $g \geq 1$ odd, the sublocus ${\mathcal M}_{(g,4)}$ of ${\mathcal M}_{g}$ consisting 
of those classes of Riemann surfaces admitting a conformal involution with exactly four fixed points is a suborbifold of dimension $(3g-1)/2$.
In this case, each connected component of ${\mathcal M}^{hyp}_{(g,4)}:={\mathcal M}_{(g,4)} \cap {\mathcal M}_{g}^{hyp}$ can be identified with the moduli space of $(g+3)$-marked sphere, which has dimension $g$.

Let ${\mathcal M}_{(g;2,2)}$, $g \geq 1$, be the moduli space of Riemann orbifolds of genus $g$ with exactly two cone points of order two. This is a complex orbifold of dimension $3g-1$. For $g \geq 2$, we let ${\mathcal M}^{hyp}_{(g;2,2)}$ its suborbifold consisting of the conformal classes of those Riemann orbifolds whose underlying Riemann surface is hyperelliptic and whose hyperelliptic involution permutes the two cone points (it does not fixes them). This space has dimension $2g$.

Finally, (see Section \ref{Sec:moduli}) ${\mathcal H}_{n}$ denotes the moduli space of generalized Humbert curves of type $n \geq 4$. As the generalized Humbert group of type $n$ is unique, then there is a natural holomorphic embedding of ${\mathcal H}_{n}$ into ${\mathcal M}_{g_{n}}$. Moreover, this moduli space is isomorphic to the 
moduli space ${\mathcal M}_{0,n+1}$ of $(n+1)$-marked spheres.

\subsection{A relation between the above parametrizing spaces}

\begin{theo}\label{main}
\mbox{}
\begin{itemize}
\item[(1)] If $n \geq 4$ is an even integer, then 
\begin{itemize}
\item[(1.1)] there is a generically injective holomorphic map
${\mathcal M}_{0,n+1} \to \left({\mathcal M}^{hyp}_{(n-2,2)}\right)^{n(n+1)/2}$. 
\item[(1.2)] there is a degree $n(n+1)/2$ holomorphic surjective map
${\mathcal M}^{hyp}_{(n-2,2)} \to {\mathcal M}_{0,n+1}$.
\item[(1.3)] there is a generically injective holomorphic map
${\mathcal M}_{0,n+1} \to \left({\mathcal M}^{hyp}_{((n-2)/2;2,2)}\right)^{n+1}$.
\item[(1.4)] there is a degree $(n+1)$ holomorphic surjective map
${\mathcal M}^{hyp}_{((n-2)/2;2,2)} \to {\mathcal M}_{0,n+1}$.
\end{itemize}
\item[(2)] If $n \geq 5$ is an odd integer, then
\begin{itemize} 
\item[(2.1)] there is a generically injective holomorphic map
${\mathcal M}_{0,n+1} \to \left({\mathcal M}^{hyp}_{(n-2,4)}\right)^{n(n+1)/2}$.
\item[(2.2)] there is a degree $n(n+1)/2$ holomorphic surjective map
${\mathcal M}^{hyp}_{(n-2,4)}\to {\mathcal M}_{0,n+1}$.

\end{itemize}
\end{itemize}
\end{theo}

\subsection{Proof of part (1) of Theorem \ref{main}}
We assume $(S,H)$ is a generalized Humbert pair of type $n \geq 4$ even and let $K_{1}$,..., $K_{n(n+1)/2}$ be those subgroups of $H$ isomorphic to ${\mathbb Z}_{2}^{n-2}$ and acting freely on $S$. Denote, as before, by $a_{1},...,a_{n+1}$ the standard generators of $H$. We already know  that $S/K_{i}$ is a hyperelliptic Riemann surface of genus $n-2$, that $H/K_{i} <Aut(S/K_{i})$ is generated by the hyperelliptic involution $j_{i}$ and a conformal involution $\tau_{i}$ with exactly two fixed points ($j_{i} \tau_{i}$ also has exactly two fixed points).  Part (1) of the following lemma asserts that, up to isomorphisms, in the above we obtain all possible pairs $(R,G)$, where $R$ runs over the hyperelliptic Riemann surfaces of genus $n-2$ and ${\mathbb Z}_{2}^{2} \cong G<Aut(R)$ contains  the hyperelliptic involution of $R$.

\begin{lemm}\label{embed2}
Let $R$ be a hyperelliptic Riemann surface of genus $n-2$, where $n \geq 4$ is even, whose hyperelliptic involution is $j$.
\begin{itemize}
\item[(1)] If $G<Aut(R)$ is so that $G \cong {\mathbb Z}_{2}^{2}$ contains $j$, then there is a generalized Humbert pair $(S,H)$ and a subgroup
${\mathbb Z}_{2}^{n-2} \cong K<H$ acting freely on $S$ so that 
$(R,G)$ is conformally equivalent to $(S/K,H/K)$.

\item[(2)]  If $u,v \in Aut(R)$ are conformal involutions, both of them different from $j$, then $\langle u,j\rangle=\langle v,j\rangle$.
\end{itemize}
\end{lemm}
\begin{proof}
(1) As $G$ contains the hyperelliptic involution, by the Riemann Hurwitz formula, the quotient $R/G$ has genus zero and exactly $n+1$ cone points, each one of order two. Now, just take $(S,H)$ as a generalized Humbert pair such that $S/H=R/G$ and use the fact that $S$ is the highest abelian regular branched cover of the orbifold $S/H$.

(2). Let us consider a $2$-fold branched cover $\pi:R \to \widehat{\mathbb C}$ (its deck group is generated by the hyperelliptic cinvolution). Then, both $u$ and $v$ descends by $\pi$ to commuting conformal involutions, say $\widehat{u}$ and $\widehat{v}$, respectively. If $\widehat{u}=\widehat{v}$, then we are done. Let us assume we have $\widehat{u} \neq \widehat{v}$, that is, $\langle  \widehat{u}, \widehat{v} \rangle \cong {\mathbb Z}_{2}^{2}$. Up to a Moebius transformation, we may assume
$\widehat{u}(z)=1/z$ and $\widehat{v}(z)=-z$. As we are assuming that $j \notin \{u,v,uv\}$, none of $u$, $v$ or $uv$ may have a common fixed point with $j$ (this because the stabilizer of any point in $Aut(R)$ is cyclic). It follows that none of $\widehat{u}$, $\widehat{v}$ or $\widehat{u}\widehat{v}$ fixes a branch value of $\pi$ and, in particular, that $R$ must have a curve representation as follows
$$y^{2}=\prod_{j=1}^{(n-1)/2} \left( x^{2}-a_{j}^{2} \right) \left( x^{2}-a_{j}^{-2}\right)$$
and $n$ is odd, a contradiction to the fact that $n$ was assumed to be even.
\end{proof}

\subsubsection{Proof of Parts (1.1) and (1.2)}
As the generic orbifold $S/H$ has trivial group of orbifold automorphisms, Lemma \ref{teo5} asserts that the $n(n+1)/2$ pairs
$$(S/K_{1},H/K_{1}),..., (S/K_{n(n+1)/2},H/K_{n(n+1)/2})$$
are generically pairwise conformally non-equivalent. Now, part (2) of Lemma \ref{embed2} asserts that the hyperelliptic Riemann surfaces $S/K_{1},..., S/K_{n(n+1)/2}$ are generically pairwise conformally non-equivalent, in particular, 
$${\mathcal M}_{0,n+1} \to \left({\mathcal M}^{hyp}_{(n-2,2)}\right)^{n(n+1)/2}:
[(S,H)] \mapsto ([S/K_{1},H/K_{1})],..., [(S/K_{n(n+1)/2},H/K_{n(n+1)/2})])$$
is a generically injective holomorphic map. This provides Part (1.1) of Theorem \ref{main}. 

Part (1.2) of Theorem \ref{main} will be just a consequence of Part (1.1) and Part (1) of Lemma \ref{embed2}.
We proceed to describe the desired surjective holomorphic map in terms of $V_{n}$. Assume we are given a hyperelliptic Riemann surface $R$ of genus $(n-2)$, whose hyperelliptic involution is $j$, and $G=\langle j,\tau\rangle \cong {\mathbb Z}^{2}$, a group of conformal automorphism of $R$, so that $\tau$ has exactly two fixed points ($j \tau$ also has exactly two fixed points) and $R/G$ is an orbifold of genus zero and exactly $n+1$ cone points, each one of order two. We may assume $R/G$ is the Riemann sphere and the conical points to be $\infty$, $0$, $1$, $\lambda_{1}$,..., $\lambda_{n-2}$, so that (i) $\lambda_{n-3}$ is the projection of both fixed points of $\tau$ and (ii) $\lambda_{n-2}$ is the projection of both fixed point of $j \tau$. This choice is not unique as we may compose at the left by a M\"obius transformation that sends any of three points in $\{\infty,0,1,\lambda_{1},...,\lambda_{n-4}\}$ to $\infty$, $0$ and $1$. This  corresponds to the action on $V_{n}$ by the subgroup ${\mathfrak S}_{n-1}=\langle s,b\rangle < {\mathfrak S}_{n+1},$
where
$$s(\lambda_{1},...,\lambda_{n-2})= \left( \frac{\lambda_{n-4}}{\lambda_{n-4}-1},\frac{\lambda_{n-4}}{\lambda_{n-4}-\lambda_{1}},..., \frac{\lambda_{n-4}}{\lambda_{n-4}-\lambda_{n-5}},
\frac{\lambda_{n-4}}{\lambda_{n-4}-\lambda_{n-3}},\frac{\lambda_{n-4}}{\lambda_{n-4}-\lambda_{n-2}} \right).$$

Next, as we may permut the involutions $\tau$ and $j\tau$, we also need to consider the action of the involution
$$c(\lambda_{1},...,\lambda_{n-2})=(\lambda_{1},...,\lambda_{n-4},\lambda_{n-2},\lambda_{n-3}) \in {\mathfrak S}_{n+1}.$$

Note that $cs=sc$ and $cb=bc$, so $\langle {\mathfrak S}_{n-1},c\rangle= {\mathfrak S}_{n-1} \oplus {\mathbb Z}_{2}$. A model of the space ${\mathcal M}^{hyp}_{(n-2,2)}$ is, by the above and Lemma \ref{embed2}, given by 
$V_{n}/({\mathfrak S}_{n-1} \oplus {\mathbb Z}_{2}).$

Also, a model of the moduli space of pairs $(R,\tau)$, where $R$ is a hyperelliptic  Riemann surface of genus $n-2$ and $\tau:R \to R$ is a conformal involution with exactly two fixed points, is given by $V_{n}/{\mathfrak S}_{n-1}$. In these models, the surjective holomorphic map in Part (1.2) of Theorem \ref{main} corresponds to the canonical projection map 
$$V_{n}/({\mathfrak S}_{n-1} \oplus {\mathbb Z}_{2}) \to V_{n}/{\mathfrak S}_{n+1}$$
in the following diagram

$$
\xymatrixcolsep{4pc}
\xymatrix{ 
V_{n} \ar[r]^{{\mathfrak S}_{n-1}} & V_{n}/{\mathfrak S}_{n-1} \ar[r]^{{\mathfrak S}_{n-1} \oplus {\mathbb Z}_{2}} \ar[rd]^{n(n+1)} &V_{n}/({\mathfrak S}_{n-1} \oplus {\mathbb Z}_{2}) \ar[d]^{\frac{n(n+1)}{2}} \\ 
& &{\mathcal M}_{0,n+1}}
$$

\begin{exem}[$n=4$]
 If $(\lambda_{1},\lambda_{2}) \in V_{4}$ are so that $S/H$ is conformally equivalent to the orbifold provided by $\widehat{\mathbb C}$ with conical points $\infty$, $0$, $1$, $\lambda_{1}$ and $\lambda_{2}$. Choose the conical points $\lambda_{1}$ y $\lambda_{2}$ and set $P(z)=(\lambda_{1}z^{2}+\lambda_{2})/(z^{2}+1)$. Then $P:\widehat{\mathbb C} \to \widehat{\mathbb C}$ is the branched covering of degree two with cover group generated by $\eta(z)=-z$ and branch values at $\lambda_{1}$ and $\lambda_{2}$. In this case $P^{-1}(\infty)=\pm i$, $P^{-1}(0)=\pm i \sqrt{\lambda_{2}/\lambda_{1}}$ and $P^{-1}(1)=\pm i \sqrt{(\lambda_{2}-1)/(\lambda_{1}-1)}$. These $6$ points define the hyperelliptic curve
$$C_{\lambda_{1},\lambda_{2}}: \; y^{2}=(x^{2}+1)\left(x^{2}+\lambda_{2}/\lambda_{1}\right)\left(x^{2}+(\lambda_{2}-1)/(\lambda_{1}-1)\right).$$

The curve $C_{\lambda_{1},\lambda_{2}}$ is one of the $10$ genus two Riemann surfaces uniformized by one of the acting freely subgroups $K_{j}$. The action of ${\mathfrak S}_{3} \oplus {\mathbb Z}_{2}$ at this level is given by:

$$
s: \; C_{\lambda_{1},\lambda_{2}} \mapsto C_{\frac{1}{1-\lambda_{1}},\frac{1}{1-\lambda_{2}}}:\; y^{2}=(x^{2}+1)\left(x^{2}+\frac{\lambda_{1}-1}{\lambda_{2}-1}\right)\left(x^{2}+\frac{\lambda_{2}(\lambda_{1}-1)}{\lambda_{1}(\lambda_{2}-1)}\right)
$$
$$
b: \; C_{\lambda_{1},\lambda_{2}} \mapsto C_{\frac{1}{\lambda_{1}},\frac{1}{\lambda_{2}}}:\; y^{2}=(x^{2}+1)\left(x^{2}+\frac{\lambda_{1}}{\lambda_{2}}\right)\left(x^{2}+\frac{\lambda_{1}(\lambda_{2}-1)}{\lambda_{2}(\lambda_{1}-1)}\right)
$$
$$
c: \; C_{\lambda_{1},\lambda_{2}} \mapsto C_{\lambda_{2},\lambda_{1}}: \; y^{2}=(x^{2}+1)\left(x^{2}+\frac{\lambda_{1}}{\lambda_{2}}\right)\left(x^{2}+\frac{\lambda_{1}-1}{\lambda_{2}-1}\right)
$$
\end{exem}

\subsubsection{Proof of Parts (1.3) and (1.4)}
 As see in Section \ref{hiper},  any subgroup $L<H$ isomorphic to ${\mathbb Z}_{2}^{n-1}$ that contains some $K_{k}$ is of the form 
$L=\langle K_{k},a_{j}\rangle$,
for some standard generator $a_{j}$ of $H$. Up to permutation of indices, we may assume
$K_{k}=\langle a_{1}a_{2}, a_{1}a_{3},...,a_{1}a_{n-1}\rangle.$ 
If $j \in \{1,2,...,n-1\}$, then
$L=\langle K_{k},a_{j} \rangle=\langle a_{1},a_{2},...,a_{n-1}\rangle$ and $H/L$ is the cyclic group generated by the hyperelliptic involution of $S/K_{k}$. We call any of these kind of subgroups $L$ a {\it hyperelliptic-${\mathbb Z}_{2}^{n-1}$-subgroups of $H$}. The following is now clear.

\begin{theo}\label{teo8}
If $(S,H)$ is a generalized Humbert pair of type $n \geq 4$ even, then the number of different hyperelliptic-${\mathbb Z}_{2}^{n-1}$-subgroups of $H$ is $n(n+1)/2$.
\end{theo}

Let us now consider the case $j \in \{n,n+1\}$.
The two different groups 
$L_{1}=\langle K_{k},a_{n} \rangle$ and $L_{2}=\langle K_{k},a_{n+1} \rangle$
have the property that $H/L_{j}$ is generated by a conformal involution (different from the hyperelliptic one) of $S/K_{k}$ having exactly $2$ fixed points. In this way, $S/L_{j}$ is an orbifold of signature $((n-2)/2;2,2)$. We call these kind of groups $L_{j}$ a {\it non-hyperelliptic-${\mathbb Z}_{2}^{n-1}$-subgroups of $H$}. At this point, we note that, as there are exactly $n(n+1)/2$ different possibilities for $K_{k}$, there are at most $n(n+1)$ different non-hyperelliptic-${\mathbb Z}_{2}^{n-1}$-subgroups of $H$.

\begin{lemm}\label{lemma2}
Let $(S,H)$ be a generalized Humbert pair of type $n \geq 4$ even and let $a_{1}$,..., $a_{n+1}$ be the standard generators of $H$.
Let $j_{1},...,j_{n-1},k_{1},...,k_{n-1} \in \{1,2,...,n+1\}$ so that
$j_1,...,j_{n-1}$ are pairwise different, $k_1,...,k_{n-1}$ are also pairwise different. 
Let $U_{1}=\langle a_{j_1}a_{j_2}, a_{j_1}a_{j_3},...,a_{j_1}a_{j_{n-1}}\rangle$ and  
$U_{2}=\langle a_{k_1}a_{k_2}, a_{k_1}a_{k_3},...,a_{k_1}a_{k_{n-1}}\rangle$
If $a_{r} \in \{1,...,n+1\}-\{j_{1},...,j_{n-1},k_{1},...,k_{n-1}\}$, then
$\langle U_{1},a_{r}\rangle=\langle U_{2},a_{r}\rangle$.
\end{lemm}
\begin{proof}
We may assume, up to permutation of indices, that
$U_{1}=\langle a_{1}a_{2}, a_{1}a_{3},...,a_{1}a_{n-1}\rangle$ and $r=n+1$. As 
$a_{1}a_{n}a_{n+1}=(a_{1}a_{2})(a_{1}a_{3})\cdots(a_{1}a_{n-1}) \in U_{1},$
$a_{1}a_{n} \in \langle U_{1},a_{n+1}\rangle$. It follows that $a_{i}a_{j} \in \langle U_{1},a_{n+1}\rangle$, for all $i,j \in \{1,2,...,n\}$. This ensures $\langle U_{2}, a_{n+1}\rangle < \langle U_{1},a_{n+1}\rangle$ and, in particular, that they are equal.
\end{proof}

As consequence of the previous Lemma, we obtain.

\begin{theo}\label{teo9}
Let $(S,H)$ be a generalized Humbert pair of type $n \geq 4$ even. Then,
there are exactly $n+1$ different non-hyperelliptic-${\mathbb Z}_{2}^{n-1}$-subgroups of $H$.
\end{theo}

Now, let $L_{1},..., L_{n+1}<H$ the $(n+1)$ different non-hyperelliptic-${\mathbb Z}_{2}^{n-1}$-subgroups of $H$. Again, as for generic pair $(S,H)$ we have that $S/H$ has trivial orbifold automorphism group, generically the $(n+1)$ orbifolds $S/L_{1}$,..., $S/L_{n+1}$ (each one of signature $((n-2)/2;2,2)$) are pairwise conformally non-equivalent. In particular, it follows that 
$${\mathcal M}_{0,n+1} \to \left({\mathcal M}^{hyp}_{((n-2)/2;2,2)}\right)^{n+1}:
[(S,H)] \to ([S/L_{1},H/L_{1}],...,[S/L_{n+1},H/L_{n+1}])$$
is a generically injective holomorphic map, obtaining Part (1.3) of Theorem \ref{main}.
As a generalized Humbert curve is the homology covering of an orbifold of genus zero with all of its cone points of order two, it follows Part (1.4) of Theorem \ref{main}.

\begin{rema}
In order to get equations for the underlying hyperelliptic Riemann surfaces $S/L_{j}$, we only need to choose one of the conical points of $S/H$ and consider the hyperelliptic Riemann surface determined by the other $n$ conical points. For example, if 
$n=4$ and $(\lambda_{1},\lambda_{2}) \in V_{4}$, then, up to equivalence, the $n+1=5$ curves of genus one are given by
$$\begin{array}{c}
y^{2}=x(x-1)(x-\lambda_{1}(\lambda_{2}-1)/\lambda_{2}(\lambda_{1}-1)), \quad
y^{2}=x(x-1)(x-(\lambda_{2}-1(/(\lambda_{1}-1)), \\
y^{2}=x(x-1)(x-\lambda_{1}/\lambda_{2}), \quad
y^{2}=x(x-1)(x-\lambda_{1}),\quad
y^{2}=x(x-1)(x-\lambda_{2}).
\end{array}
$$
\end{rema}

\subsection{Proof of part (2) of Theorem \ref{main}}
Let us now assume $(S,H)$ is a generalized Humbert pair of type $n \geq 5$ odd and that $K_{1}$,..., $K_{n(n+1)/2}$ are those subgroups isomorphic to ${\mathbb Z}_{2}^{n-2}$ acting freely on $S$. We may proceed as in the even case and to obtain the commutative diagram
$$
\xymatrixcolsep{4pc}
\xymatrix{ 
 V_{n} \ar[r]^{{\mathfrak S}_{n-1}} \ar[rd]^{{\mathfrak S}_{n+1}} &V_{n}/{\mathfrak S}_{n-1} \ar[d]^{n(n+1)} \\ 
 &{\mathcal M}_{0,n+1}}
$$
\noindent
where $V_{n}/{\mathfrak S}_{n-1}$ is a model for the moduli space of hyperelliptic Riemann surfaces admitting a conformal involution with exactly $4$ fixed points. The proofs of Parts (2.1) and (2.2) follows the same lines as the previous cases.



\begin{thebibliography}{99}

\bibitem{Accola}
Accola, R. 
{\it Topics in the Theory of Riemann Surfaces}. 
Springer-Verlag, 1994.

\bibitem{Baker}
Baker, H. F. 
{\it An introduction to the theory of multiply periodic functions}.
 Cambridge Univ. Press, 1907.


\bibitem{CGHR}
Carocca, A., Gonz\'alez, V., Hidalgo, R. A. and Rodr\'{\i}guez, R.
Generalized Humbert Curves.
{\it Israel Journal of Mathematics} {\bf 64}, No. 1 (2008), 165--192.

\bibitem{CHQ}
Carvacho, M.,  Hidalgo, R. A. and Quispe, S.
Jacobian variety of generalized Fermat curves.
{\it Quarterly Journal of Math.} {\bf 67} (2016), 261--284.


\bibitem{Edge}
Edge, W. L. 
Humbert's plane sextics of genus $5$. 
{\it Math. Proc. Cambridge Phil. Soc.} {\bf 47} No. 3 (1951), 483--495.

\bibitem{Edge2}
Edge, W. L.
The common curve of quadrics sharing a self-polar simplex.
{\it Ann. Mat. Pura Appl.} {\bf 114} (1977), 241--270.

\bibitem{GHL}
Gonz\'alez-Diez, G.,  Hidalgo, R. A. and Leyton, M.
 Generalized Fermat Curves.
{\it Journal of Algebra} {\bf 321} (2009), 1643--1660.

\bibitem{Hidalgo}
Hidalgo, R. A. 
Homology closed Riemann surfaces.
{\it Quarterly Journal of Math.} {\bf 63} (2012), 931--952.

\bibitem{Hidalgo:bases}
Hidalgo, R. A. 
Holomorphic differentials of generalized Fermat curves.
https://arxiv.org/abs/1710.01349


\bibitem{HKLP}
Hidalgo, R. A.,  Kontogeorgis, A.,  Leyton-Alvarez, M. and Paramantzouglou, P.
Automorphisms of generalized Fermat curves.
{\it Journal of Pure and Applied Algebra} {\bf 221} (2017), 2312--2337

\bibitem{Humbert}
Humbert, G.
Sur un complexe remarquable de coniques. 
{\it Jour. d'\'Ecole Polytechnique} {\bf 64} (1894), 123--149.

\bibitem{Maclachlan}
Maclachlan, C.
Abelian groups of automorphisms of compact Riemann surfaces.
{\it Proc. London Math. Soc.} (3) {\bf 15} (1965), 699--712. 


\bibitem{Mumford}
Mumford, D.,  Fogarty, J. 
{\it Geometric invariant theory}. 
Second edition. Ergebnisse der Mathematik und ihrer Grenzgebiete [Results in Mathematics and Related Areas], {\bf 34}. 
Springer--Verlag, Berlin, 1982. xii+220 pp. ISBN: 3--540--11290--1. 


\bibitem{Nag}
Nag, S.
{\it The complex analytic theory of Teichm\"uller spaces.}
A Wiley-Interscience Publication. John Wiley \& Sons, Inc., New York 1988.

\bibitem{Noether1}
Noether, E. 
Der Endlichkeitssatz der Invarianten endlicher Gruppen.
{\it Math. Ann.} {\bf 77} (1916), 89--92.




\bibitem{Royden}
Royden, H. L. 
Automorphisms and isometries of Teichmueller space.
In Lars V. Ahlfors et al., editor, Advances in the Theory of Riemann Surfaces, Vol. {\bf 66} of Ann. of Math. Studies (1971), 369--383. 

\bibitem{Saint}
Saint-Donat, B. 
On Petri's analysis of the linear system of quadrics through a canonical curve.
{\it Math. Ann.} {\bf 206} (1973), 157--175.

\bibitem{Varley}
Varley, R.
Weddle's surfaces, Humbert's curves and a certain 4-dimensional abelian variety.
{\it Amer. J. of Math.} {\bf 108} (1986), 931--952.


\end{thebibliography}
\end{document}